\documentclass[11pt,oneside,reqno]{amsart}

\usepackage[margin=2cm]{geometry}
\usepackage{amsmath,amsthm,amsfonts,amssymb,amscd,cite,graphicx}
\usepackage{latexsym}
\usepackage{enumitem}

\newtheorem{theorem}{Theorem}[section]
\newtheorem{proposition}[theorem]{Proposition}
\newtheorem{lemma}[theorem]{Lemma}
\newtheorem{definition}[theorem]{Definition}

\newcommand{\conv}{\mathrm{conv}}
\newcommand{\cp}{\mathcal{P}}
\newcommand{\zz}{\mathbb{Z}}
\newcommand{\rr}{\mathbb{R}}

\begin{document}

\baselineskip=0.20in

\title{The Convex Hull of Parking Functions of Length $n$}

\author{Aruzhan Amanbayeva$^\ast$ and Danielle Wang$^\dag$}

\address{$^\ast, ^\dag$ Department of Mathematics, Massachusetts Institute of Technology, Cambridge, MA 02139-4307, USA}
\email{$^\ast$ aruzhan@mit.edu, $^\dag$ diwang@mit.edu}

\setcounter{page}{1} \thispagestyle{empty}

\begin{abstract}  
    Let $\cp_n$ be the convex hull in $\mathbb{R}^n$ of all parking
    functions of length $n$. Stanley found the number of vertices and
    the number of facets of $\cp_n$. Building upon these results, we
    determine the number of faces of arbitrary dimension, the volume,
    and the number of integer points of $\cp_n$. 

  \bigskip



\end{abstract}

\maketitle

\section{Introduction}
Let $S$ be a finite subset of $\zz^n\subset\rr^n$. When $S$ has a
combinatorial definition, there has been a lot of interest in
understanding the convex hull $\cp=\conv(S)$ in $\rr^n$. We can ask
for such information as the $f$-vector of $\cp$ (which encodes the
number of faces of each dimension), the volume, the Ehrhart polynomial
(which counts integer points in the dilation $m\cp$ where $m$ is a
positive integer), the toric $h$-vector, etc. A prototypical example
is given by taking $S$ to consist of all permutations
$(a_1,a_2,\dots,a_n)$ of $1,2,\dots,n$. Then $\conv(S)$ is the
\emph{permutohedron}, greatly generalized by Postnikov \cite{P}.

Here we take $S$ to consist of all parking functions of length $n$.
Let $\alpha=(a_1,a_2,\dots,a_n)$ be a sequence of positive integers
$a_i \in \{1,2,\dots,n\}$, and let $b_1\le b_2\le \dots \le b_n$ be
the increasing rearrangement of $\alpha$. We call $\alpha$ a
\textit{parking function} if $b_i\le i$ for all $i \in
\{1,2,\dots,n\}$. There is a vast literature on parking functions and
their connections with other areas of mathematics. For an
introduction, see Yan \cite{yan}.

We introduce an $n$-dimensional polytope $\cp_n$, defined as the convex
hull in $\mathbb{R}^n$ of all parking functions of length $n$. This
will be the central mathematical object of this paper. In particular,
we will determine the $f$-vector, the volume, and the number of
integer points of this polytope. See Figure~\ref{fig:pf3} for a
projection (Schlegel diagram) of $\cp_3$. It is combinatorially
equivalent to ``half a 3-cube,'' i.e., cut a 3-cube in half by a
hyperplane whose intersection with the cube is a regular hexagon.

\begin{figure}[h]
\centering
\centerline{\includegraphics[width=4.5cm]{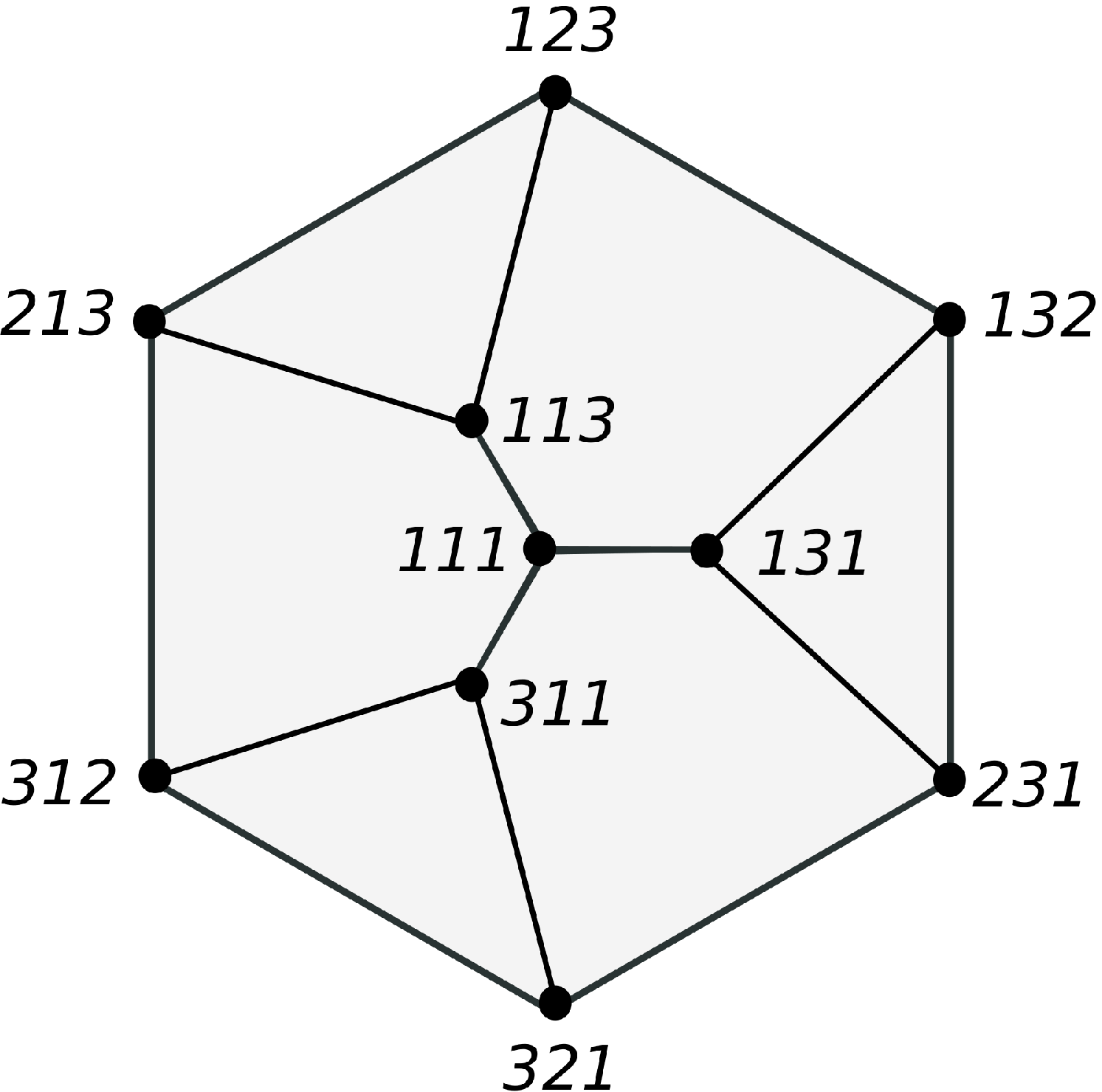}}
\caption{The polytope $\cp_3$} 
\label{fig:pf3}
\end{figure}

This paper arose from a problem proposed by Stanley in \cite{S}, which
asks to determine 
\begin{enumerate}[label=(\alph*)]
\item the number of vertices of $\cp_n$,
\item the number of $(n-1)$-dimensional faces, i.e., facets, of $\cp_n$,
\item the number of integer points in $\cp_n$, i.e., the number of
  elements of $\mathbb{Z}^n \cap \cp_n$, 
\item the $n$-dimensional volume of $\cp_n$.
\end{enumerate}

\begin{definition}
We call $F$ a face of a polytope $\cp$ if
$$F = \cp \cap \{x: c\cdot x=d\}$$
for some $c \in \mathbb{R}^n, d \in \mathbb{R}$ such that for all $x \in \cp, c\cdot x \leq d$ where the dot $\cdot$ means dot product. We call a face a vertex if it has dimension $0$, an edge if it has dimension $1$, and a facet if it has dimension $n-1$ given that $\cp$ has dimension $n$. 
\end{definition}

In a private communication with the authors, Stanley proved that the
vertices of $\cp_n$ are the permutations of 
 $$(\underbrace{1,\dots,1}_{k \text{ ones}},k+1,k+2,\dots,n),$$
for $1 \leq k \leq n$. This is proven in two parts. First, consider a
parking function $\alpha=(a_1,\dots,a_n)$ for which there is a term
$a_i>1$ such that $(a_1,\dots,a_{i-1},a_i+1,a_{i+1},\dots,a_n)$ is
also a parking function. It can be seen that $\alpha$ is a convex
combination of two other parking functions. Second, if
$\alpha=(1,\dots,1,k+1,k+2,\dots,n)$ is a convex combination of
$\beta,\gamma \in \cp_n$, then by properties of parking functions,
$\beta=\gamma=\alpha$, meaning $\alpha$ is a vertex of $\cp_n$. 

From these observations, the number of vertices of $\cp_n$ is
$$n!\left(\frac{1}{1!}+\frac{1}{2!}+\dots+\frac{1}{n!}\right).$$

Stanley also showed that the defining inequalities of $\cp_n$ are
\begin{align*}
1 \leq & x_i \leq n, \quad 1 \leq i \leq n
\\ x_i+x_j & \leq (n-1) + n , \quad i<j
\\ x_i+x_j+x_k & \leq (n-2) + (n-1) + n , \quad i<j<k
\\ \vdots
\\ x_{i_1}+x_{i_2}+\dots+x_{i_{n-2}} & \leq 3+4+\dots+ n , \quad i_1<i_2<\dots<i_{n-2}
\\ x_1+x_2+\dots+x_n & \leq 1+2+\dots+ n.
\end{align*}
Thus, the number of facets is the number of these inequalities, which
is equal to $2^n-1$. 

From these findings arose the curiosity to find the number of faces of
specified dimensions other than $0$ (i.e., vertices) and $n-1$ (i.e.,
facets). In particular, we want to find the number of $1$-dimensional
faces, i.e., edges, and more generally, the number of $i$-dimensional
faces for $0\le i \le n-1$. These numbers constitute $\cp_n$'s
$f$-vector. We define the $f$-vector of an $n$-dimensional polytope as
the vector $(f_0, f_1, \dots, f_{n-1})$, where $f_i$ is the number of
$i$-dimensional faces of the polytope.

\subsection*{Organization of the paper}
In Section~\ref{sec:edg}, we find the number of edges of $\cp_n$ by
understanding which pairs of vertices create an edge and using the
formula of the number of vertices of $\cp_n$ mentioned above. In
Section~\ref{sec:fac}, we consider the general case of $d$-dimensional
faces of $\cp_n$, determine their structure, and derive a formula for
their number which involves Stirling numbers of the second kind. In
Section~\ref{sec:vol}, we prove that the sequence $\{V_n\}$ of volumes
of $\cp_n$ satisfies a nice recurrence relation, and then use it to
find the exponential generating function of this sequence. Lastly, in
Section~\ref{sec:lat}, we show that the set of lattice points of
$\cp_n$ can be divided into sets of lattice points of several
permutohedrons, which have a formula given by Postnikov in \cite{P}. 

\subsection*{Acknowledgements}
This paper arose from a UROP+ (a research program for undergraduates)
project at M.I.T.\ undertaken by the first author and mentored by the
second.  The first author would like to thank Prof.\ David Jerison,
Prof.\ Ankur Moitra, and Dr.\ Slava Gerovitch for organizing the UROP+
program, and the Paul E. Gray (1954) UROP Fund for generously
supporting her research. Also, thanks to Prof.\ Richard Stanley for
suggesting this problem and for his guidance along the way. Finally,
the first author would like to give thanks to Daniyar Aubekerov who
helped her with coding and provided emotional support throughout this program.

\section{Edges}\label{sec:edg}

\begin{theorem}\label{thm:edges}
The number of edges of $\cp_n$ is equal to $$\frac{n\cdot
  n!}{2}\left(\frac{1}{1!}+ \frac{1}{2!}+\dots+\frac{1}{n!}\right).$$ 
\end{theorem}

\begin{definition}
Let $x$ be a parking function which is a vertex of $\cp_n$. Then it is
a permutation of $(1, \dots , 1, k+1, k+2, \dots, n)$ for some unique
$1 \le k \le n$. We say that \emph{$x$ is on layer $n - k$}. For $x=(1,1,\dots,1)$ we say that it is on layer $0$.
\end{definition}

\begin{proposition}
If $v$ and $u$ are two vertices of $\cp_n$ such that $vu$ is
an edge, then $v$ and $u$ are either from neighboring layers
(differing by $1$) or from the same layer. 
\end{proposition}

\begin{proof}
If $vu$ is an edge, then there exists $c$
such that $c\cdot v=c\cdot u > c\cdot w$ for any vertex $w$ of
$\cp_n$ distinct from $v$ and $u$. Since $\cp_n$ is invariant
under coordinate permutation, without loss of generality, we may
assume $c_1\leq \dots \leq c_n$. 

Suppose $v$ and $u$ are $t\geq 2$ layers apart from each other, so
let $v$ be a permutation of $(1,\dots,1,k,k + 1,\dots,n)$ and let
$u$ be a permutation of $(1,\dots,1,k + t,k + t+1,\dots,n)$, where
$1\leq k<k+2\leq k+t\leq n$. Since $v$ and $u$ are the unique
permutations of $(1,\dots,1,k,k + 1,\dots,n)$ and $(1,\dots,1,k,k +
1,\dots,n)$, respectively, that maximize $c\cdot x$, then, by the
rearrangement inequality, $$v=(1,\dots,1,k,k +
1,\dots,n),\ u=(1,\dots,1,k + t,k + t+1,\dots,n),$$ and
$c_{k-1}<c_k<\dots<c_n$. If $c_{k+t-1}\geq0$, then for
$w=(1,\dots,1,k+t-1,k+t\dots,n) \in \cp_n$ which is distinct from
$v$ and $u$, we have $c\cdot w\geq c\cdot u$, a
contradiction. Otherwise, if $c_{k+t-1}<0$, we have
$c_k<\dots<c_{k+t-1}<0$, so $$c\cdot v-c\cdot
u=c_k(k-1)+c_{k+1}k+\dots+c_{k+t-1}(k+t-2)<0,$$ meaning $c\cdot v
< c\cdot u$, a contradiction. Thus, $v$ and $u$ are at most one
layer apart from each other. 
\end{proof}

\begin{proposition}\label{prop:nedg}
For each vertex $v$ of $\cp_n$, there are exactly $n$ edges of $\cp_n$
with $v$ as one of the vertices. Equivalently, $\cp_n$ is a
\emph{simple} polytope. 
\end{proposition}

\begin{proof}
Suppose $v$ is on layer $n-k$. Since $\cp_n$ is invariant under
coordinate permutation, without loss of generality, we may assume
$v=(1,\dots,1,k+1,\dots,n)$. Let $vu$ be an edge of $\cp_n$,
then there exists $c\in \mathbb{R}^n$ such that $c\cdot v=c\cdot
u > c\cdot w$ for any vertex $w$ of $\cp_n$ distinct from $v$ and
$u$. By the rearrangement inequality, $c_i \leq c_{k+1} \leq \dots
\leq c_n$ for any $1 \leq i \leq k$.

If $u$ is on the same layer as $v$, then $u$ is a permutation
of $(1,\dots,1,k+1,\dots,n)$ distinct from $v$. If $c_{k+1}\leq 0$, then changing the
$(k+1)$-st coordinate of $v$ from $k+1$ to $1$ will give another
vertex $w$ of $\cp_n$ for which $c\cdot w \geq c\cdot v$, a
contradiction. Thus, $0<c_{k+1}\leq \dots \leq c_n$. If $2 \leq k \leq
n$ and $c_i\geq 0$ for some $1 \leq i \leq k$, then changing the
$i$-th coordinate of $v$ from $1$ to $k$ will give another vertex
$w$ of $\cp_n$ for which $c\cdot w \geq c\cdot v$, a
contradiction. Thus, $c_i<0$ for $1\leq i \leq k$ if $k\geq 2$. This means for $k \ge 2$, we have $u_1=\dots=u_k=1$.

Also, we have at most one pair of equal coefficients among $c_k, \dots, c_n$. Otherwise, by interchanging the corresponding coordinate values of $v$ we would get a total of $\ge 3$ distinct vertices $x$ of $\cp_n$ (including $v$) for which $cx=cv=cu$, a contradiction. At the same time if we have no such pairs, then $c_k < c_{k+1} < \dots < c_n$, and then $cv>cu$, a contradiction. Therefore, we have exactly one pair of equal coefficients among $c_k, \dots, c_n$, and since $c_k \le \dots \le c_n$, they have to be neighboring. This means $u$ differs from $v$ by exactly
one swap of two neighboring coordinates $(j,j+1)$ where $k+1 \leq j
\leq n-1$ for $2\leq k\leq n-1$, and $k = 1 \leq j \leq n-1$ for $k=1$. Hence, there are at most $n-k-1$  same layer edges with $v$ if $2 \le k \le n-1$, at most $n-1$ same layer edges with $v$ if $k=1$, and $0$ same layer edges with $v$ if $k=n$. 

In fact, each of these edges can be achieved by choosing $c$ the following way. For $k=1$, let
$$0<c_1 <\dots<c_j=c_{j+1}< \dots < c_n \text{ for some } k=1\leq j\leq n-1.$$
For $k\geq 2$, let
$$c_1 = \dots = c_k < 0<c_{k+1}< \dots <c_j=c_{j+1}<\dots< c_n \text{ for some } k+1\leq
j\leq n-1.$$

Suppose $u$ is $1$ layer apart from $v$. Then $x=v$ is the only
permutation of $(1,\dots,1,k+1,\dots,n)$ maximizing $c\cdot x$. Then $c_i<c_{k+1}<\dots<c_n$ for any $1\leq i \leq k$. Therefore, if $k\geq 2$ and $u$ is a permutation of $(1,\dots,1,k,k+1,\dots,n)$, then $(u_{k+1},\dots,u_{n})=(k+1,\dots,n)$ and thus $(u_1, \dots, u_k)$ is one of the $k$ permutations of $(1,\dots,1,k)$. Hence, there are at most $k$ edges $vu$ with $u$ one layer above $v$ (i.e., on layer $n-k+1$) for $k \ge 2$. In fact, each of these edges can be achieved by choosing $c$ such that $c_i<0$ for indices $1\leq i \leq k$ with $u_i=1$, $c_i=0$ for the index $1\leq i \leq k$ with $u_i=k$, and $0<c_{k+1}<\dots c_n$. Note that if $k=1$, then $v$ is on the highest layer (layer $n-1$), so there are no edges $uv$ such that $u$ is $1$ layer above $v$.

Again, since $c_i<c_{k+1}<\dots<c_n$, for any $1\leq i \leq k$, we have that if $k<n$ and $u$ is a permutation of $(1,\dots,1,k+2,\dots,n)$ then it has to be exactly $(1,\dots,1,k+2,\dots,n)$. Hence, there is at most $1$ edge $vu$ such that $u$ is one layer below $v$ (i.e. on layer $n-k-1$) for $k<n$. In fact, this edge can be achieved by choosing $c$ such that $c_i<0$ for $1\leq i \leq k$ and $c_{k+1}=0$. Note that if $k=n$, then $v$ is on the lowest layer (layer $0$), so there are no edges $uv$ such that $u$ is $1$ layer below $v$.

Thus, adding up $u$-on-same-layer, $u$-layer-above, and $u$-layer-below edges $vu$, we get that for $2\leq k \leq n-1$, there are $(n-k-1)+k+1=n$ edges with
$v$ as one of the vertices. For $k=1$, there are $(n-k)+0+1=n$ edges
with $v$ as one of the vertices. For $k=n$, there are $0+k+0=n$
edges with $v$ as one of the vertices. 
\end{proof}

\begin{proof}[Proof of Theorem~\ref{thm:edges}]
By Proposition~\ref{prop:nedg}, the graph of $\cp_n$ is an $n$-regular
graph
with $$V=n!\left(\frac{1}{1!}+\frac{1}{2!}+\dots+\frac{1}{n!}\right)$$
vertices. Therefore, $\cp_n$ has $\frac{nV}{2}$ edges. 
\end{proof}

\section{Faces of higher dimensions}\label{sec:fac}

In this section, we generalize this approach to understand the nature
of faces of higher dimension. More specifically, we will prove the
following theorem. 

\begin{theorem}\label{thm:faces}
Let $f_{n-s}$ be the number of $(n-s)$-dimensional faces of $\cp_n$ for $s$ from $0$ to $n$. Then,
$$f_{n-s}=\sum_{m=0,m\neq1}^s \binom{n}{m} \cdot (s-m)! \cdot S(n-m+1,s-m+1),$$
where $S(n,k)$ are the Stirling numbers of the second kind.
\end{theorem}

For each $c \in \mathbb{R}^n$, let $F_c$ be the set of points $x \in
\cp_n$ such that $c \cdot x$ is maximized (for $x \in \cp_n$). Each
face of $\cp_n$ is equal to $F_c$ for some $c \in \mathbb{R}^n$. Also,
denote the set of vertices of $\cp_n$ lying in $F_c$ by $V(F_c)$. 

For each $c$, define an ordered partition $(B_{-1}, B_0, \dots, B_k)$ of $\{1,2,\dots,n\}$, where $B_{-1}$ is the set of indices $i$ such that $c_i < 0$, $B_0$ is the set of indices $i$ such that $c_i = 0$, and $B_j$ is the set of indices $i$ such that $c_i$ is the $j$-th smallest positive value among the coordinates of $c$. Let $l_j=|B_j|$ for $j=-1,0,1,\dots,k$.

\begin{lemma}
The face $F_c$ is determined by the ordered partition $(B_{-1} , B_0 \dots, B_k)$ described above. Each face of $\cp_n$ can be uniquely defined by an ordered partition $(B_{-1}, B_0, \cdots, B_k)$ that does not satisfy $l_{-1}=0, l_0=1$ or $l_{-1}=0, l_0=0, l_1 = 1$.
\end{lemma}

\begin{proof}
Consider a vertex $v$ of $\cp_n$ that maximizes $c\cdot v$. By the
rearrangement inequality and the structure of vertices of $\cp_n$, it
is clear that $v_i=1$ for $i \in B_{-1}$. Also, $(v_i)_{i \in B_0}$ is
a permutation of $(1,\dots1, j+1, \dots, l_{-1}+l_0)$ for some $j\in
[l_{-1},l_{-1}+l_0]$, and $(v_i)_{i \in B_i}$ is a permutation of
$(l_{-1}+l_0+\dots+l_{i-1}+1, l_{-1}+l_0+\dots+l_{i-1}+2, \dots,
l_{-1}+l_0+\dots+l_{i-1}+l_i)$ for each $i$ from $1$ to $k$.  

From this conclusion, if $l_{-1}=0$ and $l_0=1$, we can change the
zero coordinate of $c$ to $-1$, and the set $V(F_c)$ will not
change. Also, if $l_{-1}=0$, $l_0=0$, and $B_1=\{i\}$, we can change
the value of $c_i$ to $-1$, and $V(F_c)$ will not change. So we do not
consider $(B_{-1},B_0,\dots,B_k)$ with $l_{-1}=0$ and $l_0=1$ or
$l_{-1}=0$, $l_0=0$, and $l_1=1$. Other than that, from the conclusion
of the previous paragraph, different ordered partitions define
different $V(F_c)$'s. 
\end{proof}

\begin{lemma}
The dimension of $F_c$ is equal to $n-k-l_{-1}$.
\end{lemma}
\begin{proof}
Let $d$ be the dimension of $F_c$. Then
$d=\text{dim}(\text{aff}(V(F_c)))$. If $d=n$ then clearly $F_c=\cp_n$
and $c=0$, so indeed $n-k-l_{-1}=n=d$. Now suppose $d<n$. Then $0
\notin \text{aff}(V(F_c))$, so $\text{dim}(\text{aff}(V(F_c)\cup
\{0\}))=d+1$.It is clear that $\text{dim}(\text{aff}(V(F_c)\cup
\{0\}))$ is the dimension of the vector space $W$ spanned by the
vectors from $0$ to points in $V(F_c)$. 

For each $j$ from $1$ to $k$, consider $B_j = \{i_1, i_2, \dots,
i_{l_j}\}$. Let $V_j$ be the set of $l_j-1$ vectors $v$ in $\mathbb{R}^n$ which
are the permutations of $(1,-1,0,\dots,0)$ having $v_{i_k}=1,
v_{i_{k+1}}=-1$, for some $1 \leq k \leq l_j-1$. Also, let $V_0$ be
the set of $l_0$ vectors $e_i$ in $\mathbb{R}^n$ which are the
permutations of $(1,0,0,\dots,0)$ having value $1$ at one of the
coordinates with index $i \in B_0$. 

Take a vector $w$ from $0$ to some point of $V(F_c)$. Consider the set $S=\left(\bigcup_{i=0}^{k} V_i\right)
\cup {w}$ of $$l_0+\sum_{i=1}^{k} (l_i-1)+1= \sum_{i=0}^{k}l_i - k +
1=n-l_{-1} -k+1$$ vectors. We will prove that $S$ spans $W$. 

For any $x \in V(F_c)$, consider the vector $a=x-w-\sum_{i \in
  B_0}(x_i-w_i)e_i$. Clearly, $a_i=0$ for $i \in B_{-1} \cup B_0$, and
for each $0 < j \leq k$, if $B_j = \{i_1, i_2, \dots, i_{l_j}\}$, then
$\sum_{m=1}^{l_j}a_{i_m}=0$. Then $(a_{i_1}, a_{i_2}, \dots,
a_{i_{l_j}})$ is a linear combination of 
 $$(1,-1,0,\dots,0), (0,1,-1,0,\dots,0), \dots, (0,\dots,0,1,-1).$$
Therefore, $a$ is a linear combination of vectors in
$\bigcup_{i=1}^{k} V_i$. Thus, $x=a+w+\sum_{i \in B_0}(x_i-w_i)e_i$ is
a linear combination of vectors in $S$, so $S$ spans $W$. 

Also, $S$ is linearly independent. If it is not, then there is a
linear combination $\beta$ of vectors in $S$ such
that $$\beta=bw+\sum_{v\in S \setminus \{w\}}b_v v=0$$ and not all of
the $b_v$ and $b$ are zero. If $l_{-1}>0$, then for all $i \in
B_{-1}$, we have $0=\beta_i=bw_i$, so $b=0$. If $k>0$, then $B_1$ is
nonempty, so 
\begin{align*}
0 & = \sum_{i \in B_1}\beta_i
\\ & =\sum_{i \in B_1}\left(bw_i+\sum_{v\in S \setminus \{w\}}b_vv_i\right)
\\ & =b\sum_{i \in B_1}w_i+\sum_{i \in B_1}\sum_{v\in S \setminus \{w\}}b_vv_i
\\ & =b\sum_{i \in B_1}w_i+\sum_{v\in S \setminus \{w\}}b_v\sum_{i \in B_1}v_i
\\ & =b\sum_{i \in B_1}w_i+\sum_{v\in S \setminus \{w\}}b_v\cdot 0
\\ & =b\sum_{i \in B_1}w_i.
\end{align*}
Therefore, $b=0$. Since $d<n$, we have $l_0<n$, so either $l_{-1}>0$ or $k>0$. In both cases $b=0$. But then $(b_1,\dots,b_{n-l_{-1} -k})\neq 0$, so $\bigcup_{i=0}^{k} V_i$ is linearly dependent, which is clearly not true.

Thus, $S$ spans $W$ and is linearly independent, which means it is a basis of $W$. Thus $d+1=\text{dim}(W)=|S|=n-l_{-1} -k+1$, so $d=n-k-l_{-1}$.
\end{proof}

\begin{proof}[Proof of Theorem~\ref{thm:faces}]
To find the number $f_{n-s}$ of $(n-s)$-dimensional faces we need to find the number of different ordered partitions $(B_{-1},B_0,\dots,B_k)$ of $\{1,\dots,n\}$ such that $l_i>0$ for $i\geq 1$ and $n-s=n-k-l_{-1}$, i.e., $s=k+l_{-1}$, not satisfying $l_{-1}=0, l_0=1$ or $l_{-1}=0, l_0=0, l_1 = 1$. For convenience, we will denote $l_{-1}$ by $m$ in further computations. We have $s=k+m$, so $m$ takes values from $0$ to $s$.

For each $m$ from $0$ to $s$, we first choose $m$ elements for $B_{-}$. Then, if $l_0=0$, we partition the remaining $n-m$ elements into $k=s-m$ nonempty ordered groups. If $l_0\geq 1$, we partition the remaining $n-m$ elements into $k+1=s-m+1$ nonempty ordered groups. Thus we have the corresponding Stirling numbers of the second kind multiplied by the number of permutations of the groups because those are ordered. Note that since we do not consider $c$ with $m=l_{-1}=0$ and $l_0=1$ or $m=l_{-1}=0$, $l_0=0$, and $l_1=1$, we need to subtract the number of such partitions. So we subtract $n\cdot k! \cdot S(n-1,k)=\binom{n}{1} \cdot s! \cdot S(n-1, s) $ and $n\cdot (k-1)! \cdot S(n-1,k-1)=\binom{n}{1} \cdot (s-1)! \cdot S(n-1,s-1)$. Therefore,
\begin{align*}
f_{n-s} & =\sum_{m=0,m\neq1}^s \binom{n}{m} \cdot \left((s-m)! \cdot S(n-m,s-m) + (s-m+1)! \cdot S(n-m, s-m+1) \right)
\\ & = \sum_{m=0,m\neq1}^s \binom{n}{m} \cdot (s-m)! \cdot S(n-m+1,s-m+1).
\end{align*}
\end{proof}

To use this formula to find the number of edges of $\cp_n$, we take $n-s=1$, so $s=n-1$. Then since $S(a,a-1)=\frac{a(a-1)}{2}$ for any positive integer $a$,
\begin{align*}
f_1 & = \sum_{m=0,m\neq1}^{n-1} \binom{n}{m} \cdot (n-m-1)! \cdot S(n-m+1,n-m)
\\ & = \sum_{m=0,m\neq1}^{n-1} \binom{n}{m} \cdot (n-m-1)! \cdot \frac{(n-m+1)(n-m)}{2}
\\ & = \sum_{m=0,m\neq1}^{n-1} \frac{n! \cdot (n-m+1)}{2m!}
\\ & = \sum_{m=1}^{n-1} \frac{n! \cdot n}{2m!} - \sum_{m=2}^{n-1} \frac{n!}{2(m-1)!} +  \sum_{m=1}^{n-1} \frac{n!}{2m!}
\\ & = \frac{n}{2}\left(\sum_{m=1}^{n-1} \frac{n!}{m!}\right) - \sum_{m=1}^{n-2} \frac{n!}{2m!} + \sum_{m=1}^{n-1} \frac{n!}{2m!}
\\ & = \frac{n}{2}(V-1) + \frac{n!}{2(n-1)!}
\\ & = \frac{nV}{2},
\end{align*}
where $V$ is the number of vertices of $\cp_n$ and is equal to
$n!\left( \frac{1}{1!}+\frac{1}{2!}+\dots+\frac{1}{n!} \right)$. This
again proves Theorem~\ref{thm:edges}. 
 
 \section{Volume}\label{sec:vol}
 To find the volume of $\cp_n$, we split the polytope into
 $n$-dimensional pyramids with facets of $\cp_n$ not containing
 $I=(1,\dots,1)$ as base and point $I$ as vertex. There are $2^n-n-1$
 such pyramids. Now we will derive a recursive formula for the volume
 of $\cp_n$ as a sum of volumes of these pyramids. 

\begin{theorem}\label{thm:volume}
 Define a sequence $\{V_n\}_{n\geq 0}$ by $V_0=1$ and $V_n=\text{Vol
 }(\cp_n)$ for all positive integers $n$. Then 
 $$V_n=\frac{1}{n}\sum_{k=0}^{n-1}\binom{n}{k}\frac{(n-k)^{n-k-1}(n+k-1)}{2}V_{k}$$
for all $n \geq 2$. 
 \end{theorem}
 
 In the proof of this theorem we will use the following ``decomposition lemma''.
 \begin{proposition}[{\cite[Proposition 2]{MPA}}] \label{prop:decomp}
 Let $K_1,\dots,K_n$ be some convex bodies of $\mathbb{R}^n$ and
 suppose that $K_{n-m+1},\dots,K_n$ are contained in some
 $m$-dimensional affine subspace $U$ of $\mathbb{R}^n$. Let $MV_U$
 denote the mixed volume with respect to the $m$-dimensional volume
 measure on $U$, and let $MV_{U^{\perp}}$ be defined similarly with
 respect to the orthogonal complement $U^{\perp}$ of $U$. Then the
 mixed volume of $K_1,\dots,K_n$ 
 \begin{align*}
 MV(K_1,\dots, K_{n-m}, & K_{n-m+1}, \dots,K_n)=
 \\ & \frac{1}{\binom{n}{m}} MV_{U^{\perp}}(K_1',\dots,K_{n-m}')MV_U(K_{n-m+1},\dots,K_n),
 \end{align*}
 where $K_1',\dots,K_{n-m}'$ denote the orthogonal projections of $K_1,\dots,K_{n-m}$ onto $U^{\perp}$, respectively.
 \end{proposition}
 
 \begin{proof}[Proof of Theorem~\ref{thm:volume}] 
 Each pyramid has a base which is a facet $F$ with points of $\cp_n$ satisfying the equation $$x_{i_1}+x_{i_2}+\dots+x_{i_k}=(n-k+1)+(n-k+2)+\dots+(n-1)+n$$ for some $k \in \{1,2,\dots,n-2,n\}$ and distinct $i_1< \dots < i_k$. 

 Let $\{j_1,j_2,\dots,j_{n-k}\}=\{1,2,\dots,n\}-\{i_1,i_2,\dots,i_k\}$. Let $\cp_{n-k}'$ be the polytope containing all points $x'$ such that $x_p'=0$ for all $p \in \{i_1,i_2,\dots,i_k\}$ and for some $x \in F$, $x_p'=x_p$ for all $p \in \{j_1,j_2,\dots,j_{n-k}\}$. Then $\cp_{n-k}'$ is an $(n-k)$-dimensional polytope with the following defining inequalities:
 \begin{align*}
 1\leq & x_{j_p}' \leq n-k, \ 1 \leq p \leq n-k
 \\ x_{j_p}'+x_{j_q}' & \leq (n-k-1)+(n-k), \ 1\leq p<q \leq n-k
 \\ x_{j_p}'+x_{j_q}'+x_{j_r}' & \leq (n-k-2)+(n-k-1)+(n-k), \ 1\leq p<q<r \leq n-k
 \\ & \vdots
 \\ x_{j_{p_1}}'+x_{j_{p_2}}'+\dots+x_{j_{p_{n-k-2}}}' & \leq 3+4+\dots+(n-k), \ 1\leq p_1<p_2<\dots<p_{n-k-2}\leq n-k
 \\ x_{j_{p_1}}'+x_{j_{p_2}}'+\dots+x_{j_{p_{n-k}}}' & \leq 1+2+3+4+\dots+(n-k).
 \end{align*}
 This means $\cp_{n-k}'$ is congruent to $\cp_{n-k}$, so $\text{Vol}_{n-k}(\cp_{n-k}')=\text{Vol}_{n-k}(\cp_{n-k})=V_{n-k}$.
 
 Let $Q_k$ be the polytope containing all points $x'$ such that for all $p \in \{j_1,j_2,\dots,j_{n-k}\}$, we have $x_p'=0$, and for some $x \in F$, we have $x_p'=x_p$ for all $p \in \{i_1,i_2,\dots,i_k\}$. Then the coordinate values $(x_{i_1}',x_{i_2}',\dots,x_{i_k}')$ of vertices of $Q_k$ are the permutations of $(n-k+1,n-k+2,\dots,n)$, meaning $Q_n$ is a $(k-1)$-dimensional polytope congruent to the permutohedron of order $k$ which has $(k-1)$-dimensional volume $k^{k-2}\sqrt{k}$.
 
 Thus, $F$ is a Minkowski sum of two polytopes $\cp_{n-k}'$ and $Q_k$ which lie in two orthogonal subspaces of $\mathbb{R}^n$. Therefore, by Proposition~\ref{prop:decomp}, the $(n-1)$-dimensional volume of $F$ is equal to 
 $$\sum_{p_1,\dots,p_n=1}^{2}MV(K_{p_1},K_{p_2},\dots,K_{p_n})=V_{n-k}\cdot k^{k-2}\sqrt{k},$$
 where $K_1=\cp_{n-k}'$ and $K_2=Q_k$. Then the volume of $\text{Pyr}(I,F)$, the pyramid with $F$ as a base and $I$ as a vertex, is equal to $$\frac{1}{n}h_k\text{Vol}(F)=\frac{1}{n}h_kV_{n-k}\cdot k^{k-2}\sqrt{k},$$ where
 $$h_k=\frac{|1+\dots+1-((n-k+1) +(n-k+2)+\dots+(n-1)+n)|}{\sqrt{1+\dots+1}} =\frac{k(2n-k-1)}{2\sqrt{k}}$$
 is the distance from point $I$ to the face $F$.
 Thus,
 $$\text{Vol}(\text{Pyr}(I,F))=\frac{1}{n} \cdot \frac{k(2n-k-1)}{2\sqrt{k}}V_{n-k}\cdot k^{k-2}\sqrt{k}=\frac{1}{n} \cdot \frac{k(2n-k-1)}{2}k^{k-2}V_{n-k}.$$
Since $V_0=1$ and $V_1=0$, we get for $n\geq 2$,
\begin{align*}
V_n & = \frac{1}{n}\left(\sum_{k=1}^{n-2}\binom{n}{k}\frac{k(2n-k-1)}{2}k^{k-2}V_{n-k}\right)+\frac{1}{n} \cdot \frac{n(n-1)}{2}n^{n-2} 
\\ & = \frac{1}{n}\left(\sum_{k=2}^{n-1}\binom{n}{n-k}\frac{(n-k)(n+k-1)}{2}(n-k)^{n-k-2}V_{k}\right)+\frac{1}{n} \cdot \frac{n^{n-1}(n-1)}{2} 
\\ & = \frac{1}{n}\sum_{k=0}^{n-1}\binom{n}{k}\frac{(n-k)^{n-k-1}(n+k-1)}{2}V_{k}.
\end{align*}
\end{proof}

For $n=1,2,\dots,8$ this formula gives the volume values
$0,\frac{1}{2}$, 4, $\frac{159}{4}$, 492, $\frac{58835}{8}$, 129237,
$\frac{41822865}{16}$. 

\begin{proposition}
Let $$f(x)=\sum_{n\geq 0} \frac{V_n}{n!}x^n$$ be the exponential generating function of $\{V_n\}_{n\geq 0}$. Let $$g(x)=\sum_{n\geq 1} \frac{n^{n-1}}{n!}x^n$$ be the exponential generating function of $\{n^{n-1}\}_{n\geq 1}$. Then $$f(x)=e^{\int \frac{x(g'(x))^2}{2}}.$$
\end{proposition}

\begin{proof}
It is known that $g(x)=xe^{g(x)}$, so 
\begin{equation*}\label{eq:g}
g'(x)=e^{g(x)}+x g'(x) e^{g(x)}=\frac{g(x)}{x}+g(x)g'(x). \tag{$\ast$}
\end{equation*}
From Theorem~\ref{thm:volume},
\begin{align*}
n \cdot \frac{V_n}{n!} & =\sum_{k=0}^{n-1}\frac{(n-k)^{n-k-1}(n+k-1)}{2(n-k)!} \cdot \frac{V_k}{k!}
\\ & = \sum_{k=0}^{n-1}\frac{(n-k)^{n-k-1}(n-k+2k-1)}{2(n-k)!} \cdot \frac{V_k}{k!}
\\ & = \sum_{k=0}^{n-1}\frac{1}{2} \cdot \frac{(n-k)^{n-k}}{(n-k)!} \cdot \frac{V_k}{k!}+\sum_{k=0}^{n-1}\frac{(n-k)^{n-k-1}k}{(n-k)!} \cdot \frac{V_k}{k!}-\sum_{k=0}^{n-1}\frac{1}{2} \cdot \frac{(n-k)^{n-k-1}}{(n-k)!} \cdot \frac{V_k}{k!}.
\end{align*}
Therefore,
$$f'(x) = \frac{1}{2}g'(x)f(x)+g(x)f'(x)-\frac{1}{2x}g(x)f(x).$$
Then
$$f'(x)(1-g(x))  =\frac{1}{2x}(xg'(x)-g(x))f(x) \underset{\text{by }
  \eqref{eq:g}}{=}
\ \frac{1}{2x}xg(x)g'(x)f(x)=\frac{1}{2}g(x)g'(x)f(x),$$ 
so
  $$ f'(x)  = \frac{g(x)g'(x)f(x)}{2(1-g(x))} \underset{\text{by }
    \eqref{eq:g}}{=} \ \frac{g(x)g'(x)f(x)}{2\left(\frac{g(x)}{xg'(x)}\right)}=
   \frac{x(g'(x))^2}{2}f(x).$$ 
Thus, $f(x)=ce^{\int \frac{x(g'(x))^2}{2}}$. It is clear that $c=1$,
so $f(x)=e^{\int \frac{x(g'(x))^2}{2}}$. 
\end{proof}

\section{Lattice Points}\label{sec:lat}
In this section we determine the number of integer points in $\cp_n$.

\begin{proposition}
Let $\cp_{n,S}$ be the set of points $x$ in $\cp_n$ satisfying $x_1+\dots+x_n=S$. For each integer $S$ from $n+1$ to $\frac{n(n-1)}{2}$ there is a unique pair of positive integers $(r,k)$ such that $2\leq r \leq k+1$,
$$\underbrace{1+\dots+1}_{k \text{ ones}}+r+(k+2)+\dots+n=S,$$
and the set of vertices of $\cp_{n,S}$ is the set of permutations of $(1,\dots,1,r,k+2,\dots,n)$. For the case $S=n$, the set of vertices of $\cp_{n,n}$ is just one vertex $(1,\dots,1)$.
\end{proposition}

\begin{proof}
It is clear that if $S=n$, then the only point $x$ in $\cp_{n,S}$ satisfies $x_1=\dots=x_n=1$. For this case we can say $k=n$ and $r$ is unnecessary.

Since $1+\dots+1<1+\dots+1+n<\dots<1+2+\dots+n$,
for each $S$ from $n+1$ to $\frac{n(n-1)}{2}$ there is a unique $k\leq (n-1)$ such that
$$1+\dots+1+(k+2)+\dots+n<S\leq 1+\dots+1+(k+1)+\dots+n.$$
Then $0<S-(1+\dots+1+(k+2)+\dots+n)\leq k$, so take $$r = 1+S-(1+\dots+1+(k+2)+\dots+n)$$ for which $1<r\leq k+1.$ Then indeed $1+\dots+1+r+(k+2)+\dots+n=S$.

Suppose there is another $(r',k')$ such that $1+\dots+1+r'+(k'+2)+\dots+n=S$. If $k<k'$, then
\begin{align*}
1+\dots+1+r'+(k'+2)+\dots+n & \leq 1+\dots+1+(k'+1)+(k'+2)+\dots+n
\\ & \leq 1+\dots+1+(k+2)+\dots+n
\\ & < 1+\dots+1+r+(k+2)+\dots+n,
\end{align*}
a contradiction. Thus, $k\geq k'$. Similarly, $k' \geq k$, so $k=k'$, from where it is clear that $r=r'$.

Now we will prove that set of vertices of $\cp_{n,S}$ is the set of permutations of $(1,\dots,1,r,k+2,\dots,n)$. Let $a=(a_1,\dots,a_n)$ be a vertex of $\cp_{n,S}$. Since $\cp_{n,S}$ is invariant under coordinate permutation, we may assume $a_1\leq \dots\leq a_n$.

If there is no $1 \leq k \leq n$ such that $a_k<k$, then clearly $a_i=i$ for all $1 \leq i \leq n$. In this case $k=1$, $r=2$, and $a$ is indeed a permutation of $(1,\dots, 1,r,k+2,\dots,n)=(1,2,\dots, n)$. Otherwise, take the greatest $1 \leq k \leq n$ such that $a_k<k$. Then $a=(a_1,\dots,a_k,k+1,\dots,n)$.
\\\\
\textbf{Case 1: $a_k=a_{k-1}$.}
Suppose $c=a_m=\dots=a_k\leq k-1$ and $a_{m-1} \neq c$. Then $$c=\frac{a_m+\dots+a_k}{k-m+1}\leq \frac{m+\dots+k}{k-m+1} = \frac{m+k}{2}.$$
Suppose $c>1$. Then there exists $\epsilon >0$ such that $\epsilon \leq \frac{j-m}{2}(k-j+1)$  for each $j$ from $m+1$ to $k$. Consider $$x=(a_1,\dots,a_{m-1},a_m-\epsilon,a_{m+1},\dots,a_{k-1}, a_k+\epsilon, a_{k+1}, \dots, a_n).$$ For any $m+1\leq j \leq k$,
\begin{align*}
a_j+\dots+a_k+\epsilon & = c(k-j+1)+\epsilon
\\ & \leq \frac{m+k}{2}(k-j+1)+\frac{j-m}{2}(k-j+1)
\\ & = \frac{j+k}{2}(k-j+1)
\\ & = j+\dots+k.
\end{align*}
This means $x$ satisfies all the defining inequalities of $\cp_n$, so $x\in \cp_{n,S}$. Therefore, $$x'=(a_1,\dots,a_{m-1},a_m+\epsilon,a_{m+1},\dots,a_{k-1}, a_k-\epsilon, a_{k+1}, \dots, a_n)$$ is also in $\cp_{n,S}$ since it is just a permutation of $x$. But then $a=\frac{1}{2}x+\frac{1}{2}x'$, so $a$ is not a vertex of $\cp_{n,S}$ if $c>1$.

Therefore, $c=1$, and since $1\leq a_1\leq \dots\leq a_k=c=1$, we have $a_1=\dots=a_k=1$ and $S=1+\dots+1+(k+1)+\dots+n$, so $(r,k)=(k+1,k)$ and $a$ is indeed a permutation of $(1,\dots,1,r,k+2,\dots,n)$.
\\\\
\textbf{Case 2: $a_k>a_{k-1}$.}
Then, since $a_{k-1}\geq 1$, we have $a_k\geq 2$. Suppose $c=a_m=\dots=a_{k-1}<a_k\leq k-1$ and $a_{m-1} \neq c$. Then 
\begin{align*}
c & = \frac{a_m+\dots+a_{k-1}}{k-m}   
\\ & = \frac{a_m+\dots+a_{k-1}+a_k-a_k}{k-m}
\\ & \leq \frac{m+\dots+k-a_k}{k-m}
\\ & = \frac{\frac{1}{2}(m+k)(k-m+1)-a_k}{k-m}.
\end{align*}
Suppose $c>1$. For any $j$ from $m+1$ to $k$,
\begin{align*}
(j+\dots+k)- ( & a_j + \dots +a_k) = \frac{1}{2}(j+k)(k-j+1)-c(k-j)-a_k
\\ & \geq \frac{1}{2}(j+k)(k-j+1)-(k-j)\frac{\frac{1}{2}(m+k)(k-m+1)-a_k}{k-m}-a_k
\\ & = \frac{1}{2}(j+k)(k-j+1)-(k-j)\frac{\frac{1}{2}(m+k)(k-m+1)}{k-m} + a_k\left(\frac{m-j}{k-m}\right)
\\ & > \frac{1}{2}(j+k)(k-j+1)-(k-j)\frac{\frac{1}{2}(m+k)(k-m+1)}{k-m} + \frac{k(m-j)}{k-m}
\\ & =\frac{(k-j)(j-m)}{2} \geq 0.
\end{align*}
Then there exists $\epsilon >0$ such that $\epsilon < (j+\dots+k)-(a_j+\dots+a_k)$  for each $j$ from $m+1$ to $k$.
Consider $$x=(a_1,\dots,a_{m-1},a_m-\epsilon,a_{m+1},\dots,a_{k-1},
a_k+\epsilon, a_{k+1}, \dots, a_n).$$ For any $m+1\leq j \leq k$,
$a_j+\dots+a_k+\epsilon < j+\dots+k$. This means $x$ satisfies all the
defining inequalities of $\cp_n$, so $x\in
\cp_{n,S}$. Also, $$x'=(a_1,\dots,a_{m-1},a_m+\epsilon,a_{m+1},\dots,a_{k-1},
a_k-\epsilon, a_{k+1}, \dots, a_n)$$ is also in $\cp_{n,S}$. But then
$a=\frac{1}{2}x+\frac{1}{2}x'$, so $a$ is not a vertex of $\cp_{n,S}$
if $c>1$.

Therefore, $c=1$ and since $1\leq a_1\leq \dots\leq a_{k-1}=c=1$, we
have $a_1=\dots=a_{k-1}=1$. Then $S=1+\dots+1+a_k+(k+1)+\dots+k$,
where $2\leq r=a_k<k$, so $a$ is indeed a permutation of
$(1,\dots,1,r,k+1,\dots,n)$. 
\end{proof}

Thus, we have that $\cp_{n,S}$ is a permutohedron with permutations of
$(1,\dots,1,r,k+2,\dots,n)$ as its vertices. In the case $S=n$,
$\cp_{n,S}$ is a permutohedron consisting of one point
$(1,\dots,1)$. In other words, $\cp_{n,S}$ is the convex hull of all
permutations of vector $(x_1,\dots,x_n)$, where 
$$(x_1,\dots,x_n)=
\begin{cases} 
(1,\dots,1,r,k+2,\dots,n) & \text{if } S>n,
\\ (1, \dots, 1) & \text{if } S=n.
\end{cases}$$

Let $N(P)$ denote the number of integer points in a polytope $P$. Then,
$$N(\cp_n)=\sum_{S=n}^{\frac{n(n-1)}{2}}N(\cp_{n,S}).$$
From \cite[Section 4]{P}, $\cp_{n,S}$ is a generalized
permutohedron $\cp_{n-1}(\mathbf{Y})$ with $Y_I=y_{|I|}$ for any $I
\subset [n]$ and 
\begin{align*}
y_1 & = x_1
\\ y_2 & = x_2-x_1
\\ y_3 & = x_3-2x_2+x_1
\\ & \ \vdots
\\ y_n & = \binom{n-1}{0}x_{n}-\binom{n-1}{1}x_{n-1}+\dots \pm \binom{n-1}{n-1}x_{1}.
\end{align*}
Therefore by \cite[Theorem 4.2]{P}, we have proved the following result.

\begin{theorem} \label{thm:ip}
We have $N(\cp_n)=\sum_{S=n}^{\frac{n(n-1)}{2}}N(\cp_{n,S})$, where
 $$ N(\cp_{n,S})=\frac{1}{(n-1)!}\sum_{(S_1,\dots,S_{n-1})}\left\{Y_{S_1}\cdots
   Y_{S_{n-1}}\right\}. $$ 
The summation is over ordered collections of subsets
$S_1,\dots,S_{n-1} \subset [n]$ such that for any distinct
$i_1,\dots,i_k$, we have $|S_{i_1}\cup \dots \cup S_{i_k}|\geq k+1$,
and  
 $$\left\{
  \prod_{I}Y_I^{a_I}\right\}:=(Y_{[n]}+1)^{\{a_{[n]}\}}\prod_{I\neq
  [n]}Y_I^{\{a_I\}} \text{, where } Y^{\{a\}}=Y(Y+1)\dots
  (Y+a-1). $$
\end{theorem}

The numbers $N(\cp_n)$ for $1\leq n\leq 8$ are given by
$(1,3,17,144,1623,22804,383415,7501422)$. 

\section{Further questions}
What other properties of $\cp_n$ might be worth investigating? Here
are two possibilities.
\begin{enumerate}[label=(\alph*)]
  \item Because $\cp_n$ is a simple polytope
    (Proposition~\ref{prop:nedg}), its dual $\cp^*_n$ is
    simplicial. Thus $\cp^*_n$ has an $h$-vector
    $(h_0,h_1,\dots,h_n)$ which is a symmetric ($h_i=h_{n-i}$),
    unimodal sequence of positive integers satisfying
     $$ \sum_{i=0}^n h_i =
    n!\left(\frac{1}{1!}+\frac{1}{2!}+\dots+\frac{1}{n!}\right) , $$
    the number of facets of $\cp^*_n$ (or vertices of $\cp_n$) \cite{rs}.
    Is there a simple generating function, combinatorial formula,
    etc., for the numbers $h_i$?
  \item Is there a formula for the Ehrhart polynomial (e.g.,
    \cite[{\S}4.6.2]{ec1}) $i(\cp_n,m)$ generalizing
    Theorem~\ref{thm:ip} (the case $m=1$)?
\end{enumerate}


\begin{thebibliography}{9}
  \bibitem{rs}
     R. Stanley,
     \newblock{The number of faces of a simplicial convex polytope},
     \newblock{\emph{Advances in Math.}\ 35 (1980), 236--238.}
     \bibitem{ec1}
       R. Stanley
       \newblock{\emph{Enumerative Combinatorics, Volume 1}, second
       edition,  Cambridge University Press, Cambridge, UK (2012).}
	\bibitem{S}
		R. Stanley,
		\newblock Problem 12191,
		\newblock{PROBLEMS AND SOLUTIONS, \emph{American
                    Math.\ Monthly},} 
		\newblock 127:6 (2020), 563-571.
	\bibitem{P}
		A. Postnikov,
		\newblock Permutohedra, associahedra, and beyond,
		\newblock \emph{Int.\ Math.\ Res.\ Not.\ IMRN 2009},
                no.~6, 1026--1106.
	\bibitem{MPA}
		M. Dyer, P. Gritzmann and A. Hufnagel,
		\newblock On the complexity of computing mixed volumes,
                \newblock \emph{SIAM J. Comput.}\ 27 (1998), 356–400
                (electronic).
              \bibitem{yan} C.\,H.\ Yan,
                \newblock Parking functions,
               \newblock in \emph{Handbook of Enumerative
                 Combinatorics} (M. B\'ona, ed.), Discrete
               Math.\ Appl.\ (Boca Raton), CRC Press, Boca Raton, FL,
               2015, pp.~835--893.
            
\end{thebibliography}
\end{document}